\documentclass[12pt,reqno]{amsproc}

\usepackage[margin=.9in]{geometry}

\usepackage[T1]{fontenc}
\usepackage{bbm}
\usepackage[usenames]{color}
\usepackage{mathpazo}
\usepackage{amsmath,amssymb,amsthm}
\usepackage{wrapfig}
\usepackage{tikz-cd}

\usepackage{caption}
\usepackage{mathrsfs,bm}
\usepackage[all,cmtip]{xy}

\definecolor{PineGreen}{rgb}{0.0,0.47,0.44}
\definecolor{MidnightBlue}{rgb}{0.1,0.1,0.44}
\definecolor{magenta}{rgb}{1.0,0.0,1.0}
\definecolor{bl1}{HTML}{4479A1}
\definecolor{pur1}{HTML}{52196D}
\definecolor{mag1}{HTML}{2AD0F1}
\definecolor{org1}{rgb}{.92,.39.21}
\definecolor{pur2}{rgb}{.53,.47,.7}

\usepackage{hyperref}
\hypersetup{
	colorlinks=true,
	linkcolor=blue,
	citecolor=blue,
	filecolor=magenta,
	urlcolor=blue
}

\makeatletter
\newcommand{\eqnum}{\refstepcounter{equation}\textup{\tagform@{\theequation}}}
\makeatother

\newtheorem{theorem}{Theorem}
\numberwithin{theorem}{section}
\newtheorem{proposition}[theorem]{Proposition}
\newtheorem*{theorem*}{Theorem}
\newtheorem{lemma}[theorem]{Lemma}

\theoremstyle{definition}

\newtheorem{definition}[theorem]{Definition}

\theoremstyle{remark}
\newtheorem{remark}[theorem]{Remark}

\newcommand{\RR}{\mathbb{R}}

\newcommand{\PP}{\mathbb{P}}
\newcommand{\pp}{\mathbb{P}}
\newcommand{\CC}{\mathbb{C}}
\newcommand{\ZZ}{\mathbb{Z}}
\newcommand{\BB}{\mathbf{B}}

\DeclareMathOperator{\HG}{\mathrm{{H}}}
\newcommand{\cO}{\mathscr{O}}
\newcommand{\cM}{\mathscr{M}}

\newcommand{\Lk}{\mathscr{L}}
\newcommand{\Clk}{\Lk}

\newcommand{\LC}{\bm{\Lambda}}
\newcommand{\e}{\mathbf{e}}

\newcommand{\set}[1]{{\left\{{#1}\right\}}}
\newcommand{\ip}[1]{\left\langle{#1}\right\rangle}

\newcommand{\bW}{\mathbf{W}}
\newcommand{\bN}{\mathbf{N}}

\newcommand{\inj}{\hookrightarrow}

\usepackage[foot]{amsaddr}

\begin{document}
	
	\title{Complex Links and Hilbert-Samuel Multiplicities}
	\author{Martin Helmer} 
	\address[MH]{
		Mathematical Sciences Institute,
		The Australian National University,
		Canberra, Australia}\email{martin.helmer@anu.edu.au}
	\author{Vidit Nanda}\address[VN]{Mathematical Institute,
		University of Oxford, Oxford, United Kingdom}\email{nanda@maths.ox.ac.uk}
	
	\begin{abstract} 
		We describe a framework for estimating Hilbert-Samuel multiplicities $\e_XY$ for pairs of projective varieties $X \subset Y$ from finite point samples rather than defining equations. The first step involves proving that this multiplicity remains invariant under certain hyperplane sections which reduce $X$ to a point $p$ and $Y$ to a curve $C$. Next, we establish that $\e_pC$ equals the Euler characteristic (and hence, the cardinality) of the complex link of $p$ in $C$. Finally, we provide explicit bounds on the number of uniform point samples needed (in an annular neighborhood of $p$ in $C$) to determine this Euler characteristic with high confidence.
	\end{abstract}
	\maketitle
	
	\section{Introduction}
	
	One of the most fundamental quantities of interest in intersection theory is the {\bf Hilbert-Samuel multiplicity}, which associates an integer $\e_XY \geq 0$ to each pair consisting of an irreducible subvariety $X$ inside a pure-dimensional scheme $Y$. This integer serves -- among other things -- as a coarse measurement of the singularity type of $X$ inside $Y$. When $Y$ is reduced, $\e_XY = 1$ holds if and only if $X$ is nonempty and smoothly embedded in $Y$. The importance of Hilbert-Samuel multiplicities stems from their wide-ranging connections with several other intersection-theoretic invariants. For instance, $\e_XY$ appears as the coefficient of $[X]$ in the {Segre class} $s(X,Y)$ \cite[Chapter 4.3]{fulton2013intersection}, in  Fulton and MacPherson's intersection product \cite[Chapter 12.3]{fulton2013intersection}, and in Serre's Tor formula \cite[Theorem 1, pg 112]{serre2012local}. Computing $\e_XY$, either directly from its definition or as a consequence of these connections,  requires serious algebraic manipulations of the defining equations for $X$ and $Y$. 
	
	Our goal in this paper is to describe a new framework for estimating $\e_XY$ from  finite local point samples without recourse to any such equations.  In this setting, we have no means to capture the scheme structure of $Y$, and will therefore restrict to the case where $Y$ is reduced, i.e.,~a pure dimensional variety in some $n$-dimensional projective space $\pp^n$. The methods developed here could also be applied to any data set which we would expect to have the structure of a complex variety, even if the variety is not known. Here is an informal version of our main result for estimating $\e_XY$ for such pairs $X \subset Y$.
	
	\begin{theorem*}
		Let $L \subset \pp^n$ be a linear space obtained by intersecting $(\dim Y - 1)$ hyperplanes which are general except for the requirement that they all pass through a generic point $p$ of $X$. The Hilbert-Samuel multiplicity $\e_XY$ can be determined with high confidence from a sufficiently large (but finite) uniform point sample $S$ lying on the curve $Y \cap L$ in a local annular neighbourhood around $p$.
	\end{theorem*}
	
	\noindent We provide explicit bounds on how large $S$ must be in terms of the local geometry of $Y \cap L$ near $p$ and the desired probability of successful estimation. Our proof has three basic steps, each involving a different key ingredient and producing an intermediate result. These steps are summarized below. 
	
	\subsection*{Step 1: Algebra} We first establish that $\e_XY$ is invariant under the operation of slicing both $X$ and $Y$ by certain hyperplanes. The key ingredient here is a new {\em degree formula} for Hilbert-Samuel multiplicities \cite[Theorem 5.3]{HH19}. Using this formula, we prove the following result.
	
	\begin{theorem*} [\textbf{A}]
		Given ${X} \subset {Y}$ as above, let $L$ be the intersection of $k$ general hyperplanes which all pass through some general point $p$ of $X$; then,
		\begin{enumerate}
			\item if $k \leq \dim {X}$, then $\e_{{X}}{Y} = \e_{{X} \cap L}({Y} \cap L)$; moreover,
			\item if $k = \dim {X}$, then $\e_{{X}}{Y} = \e_p({Y} \cap L)$; and finally,
			\item if $\dim X < k \leq \dim Y - 1$, then $\e_{{X}}{Y} = \e_p({Y} \cap L)$.
		\end{enumerate}
	\end{theorem*}
	
	In fact, the first two assertions follow readily from basic properties of Segre classes whereas the last one is new and makes essential use of the aforementioned degree formula from \cite{HH19}. As a consequence of this third assertion (for $k = \dim Y - 1$), every $\e_XY$ calculation can be reduced to the case where $X = p$ is a point and $Y = C$ is a curve in $\pp^n$ containing $p$. In this special case, the degree formula for Hilbert-Samuel multiplicity simplifies to
	\[
	\e_pC = \deg(C) - \deg((C \cap H) - p),
	\]
	where $H$ is a general hyperplane passing through $p$. While this is a convenient reformulation for algebraic computation of $\e_pC$, both degrees appearing on the right side are global computations in the sense that they require checking for intersections far away from $p$. The purpose of the next step is to replace these with a local computation near $p$.
	
	\subsection*{Step 2: Topology}  
	
	The starting point for our second step is the observation that $\deg(C)$ equals the cardinality of $C \cap H'$, where $H'$ is a general hyperplane in $\pp^n$. Crucially, we let $H'$ be parallel to the plane $H$ when restricted to an affine chart of $\pp^n$ containing $p$. The special ingredient here is {\em Thom's first isotopy lemma} \cite[Chapter I.1.5]{SMTbook}, which allows us to relate $\e_pC$ to the Euler characteristic (and hence, cardinality) of the zero-dimensional space 
	\[
	\Clk_p := C \cap \BB_\epsilon(p) \cap H'.
	\]
	Here $\BB_\epsilon(p)$ denotes a small closed ball around $p$ in a chart of $\pp^n$. In particular, we employ a homological argument to show the following result.
	
	\begin{theorem*}[\textbf{B}] If $p$ is any (possibly singular) point on a curve $C \subset \pp^n$, then its Hilbert-Samuel multiplicity satisfies $\e_pC = \chi(\Clk_p)$, where $\chi$ denotes Euler characteristic.
	\end{theorem*}
	
	The space $\Clk_p$ plays a fundamental role in (complex) stratified Morse theory --- it provides {\em normal Morse data} for $p$ with respect to a stratified Morse function defined on $C$, and is called the {\bf complex link} of $p$ in $C$ \cite[Ch II.2]{SMTbook}. Variants of Theorem ({\bf B}) have been assigned as exercises to the reader on several occasions, including the Introduction to \cite{SMTbook} and \cite[Ex 4.6]{masseychar}. However, we were unable to locate a proof in the literature; since it forms an essential part of our overall argument, we have included a proof here.
	
	\subsection*{Step 3: Geometry} It remains to estimate the cardinality of $\Clk_p$ using a uniform finite point sample $S$ chosen from $B := C \cap \BB_\epsilon(p)$. The main difficulty here is that generically the intersection $S \cap H'$ will be empty even when the sample size is enormous. As such, we are compelled to search for points of $S$ which lie within some small distance $\eta > 0$ of $H'$, and hope that these points naturally organize into $\e_XY$ many clusters. The key ingredient here is a suite of {\em geometric inference results}, which date back to the work of Niyogi, Smale and Weinberger from \cite{niyogietal}. Given a compact Riemannian submanifold $M \subset \RR^d$ and a probability parameter $\gamma \in (0,1)$, these results give explicit bounds on the cardinality of a finite point sample $P \subset M$ required to estimate the homology of $M$ with probability exceeding $(1-\gamma)$.
	
	Recently, Wang and Wang have extended results of \cite{niyogietal} to the case where $M \subset \RR^d$ is a smooth submanifold with boundary \cite{wang2}; they provide an explicit lower bound $N_M(\alpha,\gamma)$ on the size of a uniform point sample $P \subset M$ required to ensure, again with probability at least $(1-\gamma)$, that $P$ is $\alpha/2$-dense in $M$. These results require $\alpha$ to be sufficiently small relative to the injectivity radii of the embeddings $M \inj \RR^d$ and $\partial M \inj \RR^d$ of the manifold and its boundary respectively. Although the space $B$ of interest to us is not a manifold (thanks to the singularity at $p$), it does become a manifold with boundary by removing the interior of a smaller ball $\BB_{\epsilon_0}(p)$ with $\epsilon_0 < \epsilon$. After this excision, we can safely apply the density results from \cite{wang2} and obtain the following result.
	
	\begin{theorem*} [\textbf{C}] There exists, for all sufficiently small radius $\alpha > 0$ and probabilities $\gamma \in (0,1)$, an explicit bound $N_M(\alpha,\gamma)$ with the following property. Any uniformly sampled subset 
		\[
		S \subset C \cap \left(\BB_\epsilon(p) - \BB_{\epsilon_0}(p)^\circ\right)
		\] of cardinality $\#S > N_M(\alpha,\gamma)$ can be used to correctly estimate the Euler characteristic $\chi(\Clk_p)$ with probability exceeding $(1-\gamma)$.
	\end{theorem*}
	
	Combining Theorems ({\bf A}), ({\bf B}) and ({\bf C}) gives the promised main result.
	
	\subsection*{Organisation} In Section \ref{section:complink} we briefly review the definition of the complex link. Section \ref{sec:algmult} focuses on Step 1; here we give a brief overview of the Hilbert-Samuel multiplicity and prove Theorem ({\bf A}). In Section \ref{sec:pointcurve} we implement Step 2 by providing a proof of the folklore Theorem ({\bf B}). And finally, in Section \ref{section:Multiplcity_From_Finite_Samples} we carefully state and establish Theorem ({\bf C}) by describing not only the precise form of the bound $N(\alpha,\gamma)$, but also the precise constraints on $\alpha$ imposed by the local geometry of $C$ near $p$.

	{\footnotesize
		\subsection*{Acknowledgements}
		
		Mark Goresky kindly shared an advance copy of his survey {\em Morse theory, stratifications and sheaves} \cite{goresky-survey} with us; that paper served as our Polaris while we navigated the formidable waters surrounding these topics. We  are grateful to Heather Harrington for organizing the event in Oxford where our collaboration was conceived, to Kate Turner for hosting VN's visit to Canberra where it reached adolescence, and to Yossi Bokor for writing helpful software during these early days. We thank the Sydney Mathematics Research Institute (SMRI) at the University of Sydney for their generous hospitality. VN's work was supported by the EPSRC grant EP/R018472/1 and by the DSTL grant D015 funded through the Alan Turing Institute.
	}

	\section{Complex Links}\label{section:complink}
	
	A {\em stratification} of a topological space $\bW$ is a filtration 
	\[
	\emptyset = W_{-1}\subset W_0 \subset \cdots \subset W_k=\bW
	\] by closed subspaces so that each consecutive difference {$W_{i}-W_{i-1}$} is a (possibly empty or disconnected) $i$-dimensional manifold called the $i$-{\em stratum}. Throughout this section, $\bW$ will denote a Whitney-stratified complex analytic subspace of $\CC^n$. We assume that each stratum $X \subset \bW$ is a connected complex analytic manifold, and write $Y > X$ to indicate that the closure of the stratum $Y$ contains the stratum $X$. We further require all pairs of strata $X < Y$ to satisfy Whitney's {Condition (B)} --- see \cite[Section 19]{whitney1965tangents}, \cite[Section 2]{Mather2012}, or \cite[Chapter I.1.2]{SMTbook}.  Let $T_pX$ denote the $(\dim X)$-dimensional linear subspace of $\CC^n$ which corresponds to the tangent space of a stratum $X$ at a point $p$ in $X$. 
	
	Fix a connected component of a stratum $X \subset \bW$ and consider an arbitrary point $p$ in $X$. Since it remains difficult to illustrate even $2$-dimensional complex varieties, the following real picture (where $n=3$ and $\dim \bW = 2$ while $\dim X=1$) will serve as a proxy for the local structure of $\bW$ near $p$.
	
	\begin{figure}[h!]
		\centering
		\includegraphics[scale=.35]{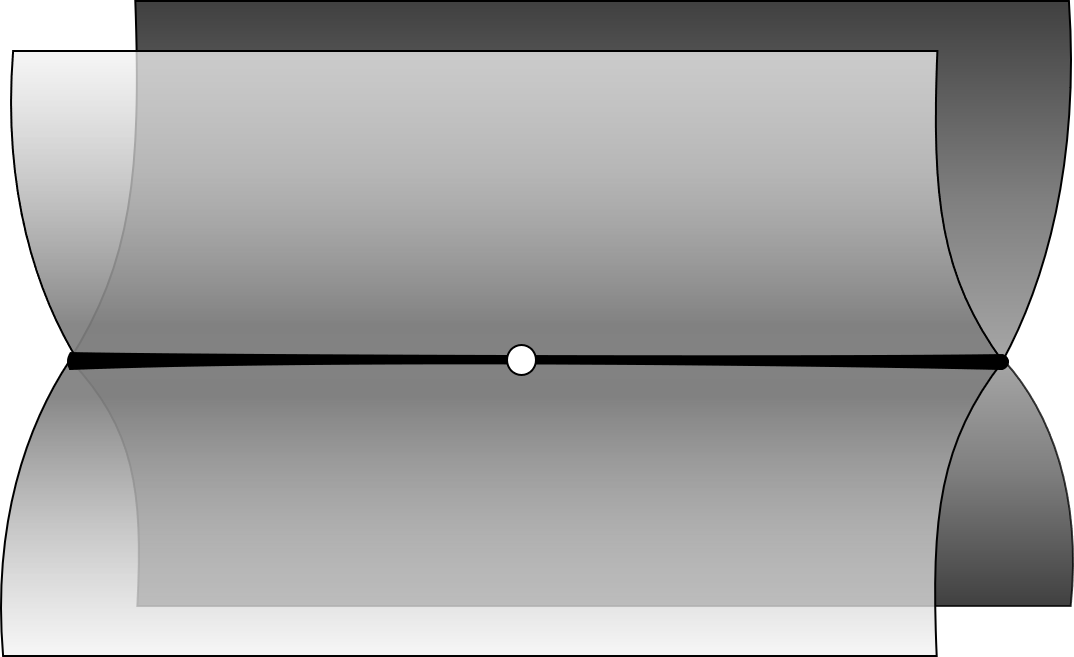}
	\end{figure}
	
	\noindent The stratum $X$ is represented by the horizontal line along which the four sheets intersect, and the chosen point $p$ is located near the center of $X$. We say that an affine subspace $A \subset \CC^n$ containing $p$ is {\em transverse} to $X$ at $p$ if the sum of tangent subspaces given by \[T_pX + T_pA  = \set{v+w \mid v\in T_pX \text{ and }w \in T_pA}\] equals $T_p\CC^n = \CC^n$. 
	
	\begin{definition} \label{def:normslice} A subset $\bN \subset \bW$ is called a {\bf normal slice} to $X$ at $p$ if it equals the intersection $\bW \cap A$ for some $(n-\dim X)$-dimensional affine subspace $A \subset \CC^n$ which intersects $X$ transversely at $p$. 
	\end{definition}
	One possible choice of $\bN$ for our example is shown below:
	
	\begin{center}
		\includegraphics[scale=.35]{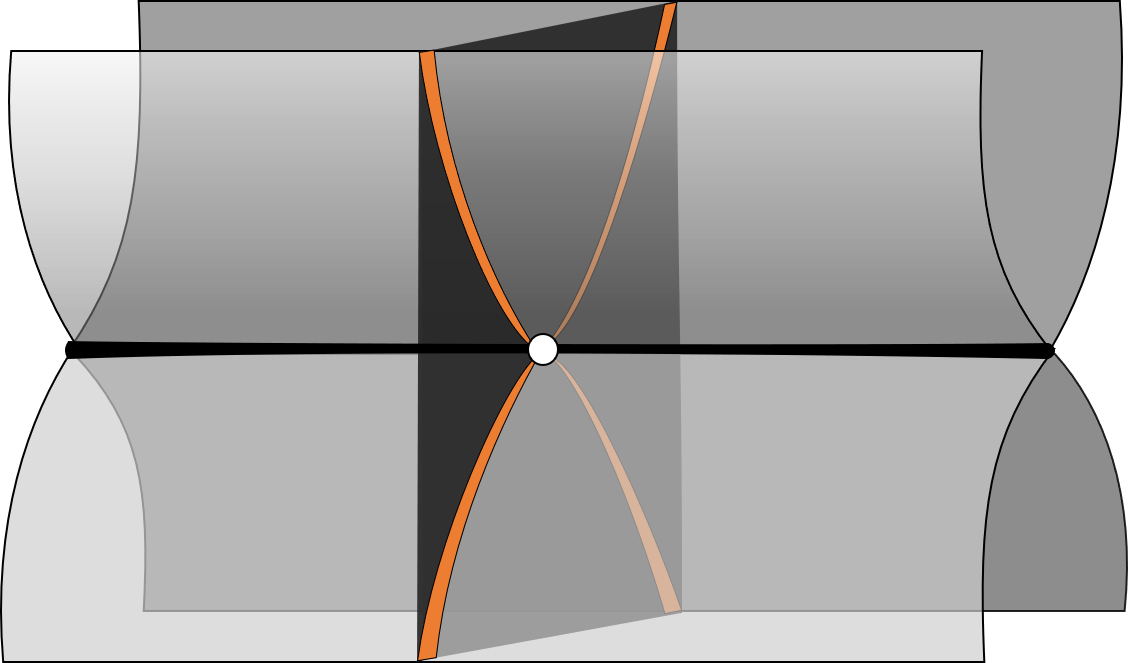}
	\end{center}
	
	\noindent Here $A$ is the plane which crosses $X$ at $p$, while $\bN$ is the union of four half-open arcs, all of which intersect at $p$. Evidently, $\bN$ will not be a manifold in general; on the other hand, it follows from the definition of a Whitney stratification that $A$ will remain transverse, at least in a small neighbourhood around $p$, to all higher strata $Y > X$. Thus, $\bN$ inherits a Whitney stratification from $\bW$ near $p$ as follows. Each $(\dim Y)$-dimensional stratum $Y > X$ carves out a (possibly disconnected, $(\dim Y-\dim X)$-dimensional) stratum $Y \cap A$ of $\bN$.  
	
	Fix a radius $\epsilon > 0$ so that the intersection of $\bN$ with the open ball $\BB_\epsilon(p)$ of radius $\epsilon$ around $p$ inherits a Whitney stratification from $\bW$ in the manner described above\footnote{\noindent More precisely, two natural transversality constraints must hold for every radius $e \leq \epsilon$ and for every stratum $Y$ of $\bW$. First, the boundary of $\BB_e(p)$ must be transverse to $Y$ in $\CC^n$, and second, the boundary of $\BB_e(p) \cap A$ must be transverse to $Y \cap A$ in $A$.}. We write $\bN_\epsilon(p) = \bN \cap \BB_\epsilon(p)$ to indicate this {\em restricted} normal slice. The next definition will make use of our chosen $\bN$ and $\epsilon$, and also of the usual inner product $\ip{\bullet,\bullet}$ defined on the ambient space $\CC^n$. 
	
	\begin{definition}\label{def:xi}
		A vector $\xi$ in the affine space $A$ is called {\bf nondegenerate} for the pair $(\bN,\epsilon)$ if the following property holds for all strata $Y > X$. Given any sequence $\{(q_i,v_i)\}$ in the tangent bundle of $Y' = Y \cap \bN_\epsilon(p)$ where $q_i$ limits to $p$, if the $v_i$ limit to some nonzero vector $v$ then $\ip{\xi,v} \neq 0$.
	\end{definition}
	
	If we restrict our pictorial example to the affine plane $A$, then the set of degenerate vectors will span the vertical line through $p$ because the orthogonal complement of this vertical line (in $A$ through $p$, as drawn below) shares a limiting tangent with all four arcs of $\bN$. For any vertically-aligned $\xi$, one can find a sequence $(q_i,v_i)$ in the tangent bundle of each arc with $q_i \to p$ and $v_i \to v \neq 0$ lying along the horizontal line, which in turn forces $\ip{\xi,v} = 0$. Any $\xi$ off the vertical line will be nondegenerate.
	\begin{center}
		\includegraphics[scale=.4]{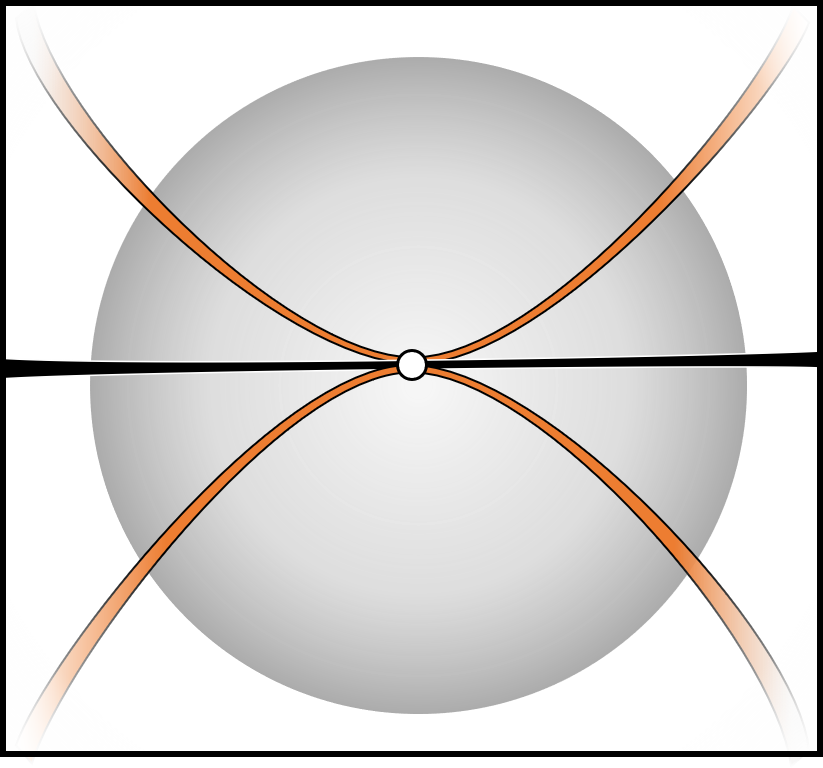}
	\end{center}
	
	\noindent Fix a nondegenerate vector $\xi$ for $(\bN,\epsilon)$, and consider the map 
	\[
	\pi_\xi:A \to \RR
	\] given by taking the real part of the affine-linear complex functional $z \mapsto \ip{z-p,\xi}$. By nondegeneracy, there exists a  $\delta > 0$ so that if the differential \[(d\pi_\xi)_q:T_qY' \to \RR\] is identically zero at some point $q \neq p$ in a stratum $Y'$ of $\bN_\epsilon(p)$, then $|\pi_\xi(q)| > \delta$ . In other words, the preceding definitions and choices have been concocted in order to ensure that $\pi_\xi$ restricts to a {\bf stratified Morse function} on  $\bN_\epsilon(p)$ as in \cite[Chapter I.2]{SMTbook}; and moreover, $p$ is its unique critical point valued in the interval $[-\delta,\delta]$. Here is a summary of all these choices that have been made for the stratum $X \subset \bW$ of dimension $\dim X$:
	\begin{enumerate}
		\item a point $p \in X$,
		\item an $(n-\dim X)$-dimensional affine subspace $A \subset \CC^n$ transverse to $X$ at $p$,
		\item a radius $\epsilon > 0$ so that $\bN_\epsilon(p) = \bW \cap A \cap \BB_\epsilon(p)$ inherits a  stratification from $\bW$,
		\item a nondegenerate vector $\xi \in A$, and finally,
		\item another radius $\delta \in (0,\epsilon)$ so that $\pi_\xi:A \to \RR$ has no critical points $q \neq p$ in $\bN_\epsilon(p)$ with $\pi_\xi(q)$ in $[-\delta,\delta]$.
	\end{enumerate}
	The following definition makes provisional use of this tuple $(p,A,\epsilon,\xi,\delta)$.
	
	\begin{definition}\label{def:clink}
		The {\bf complex link} of the stratum $X \subset \bW$ with respect to the choices $(p,A,\epsilon,\xi,\delta)$ is the intersection 
		\[
		\Clk_X = \bN_\epsilon(p) \cap \pi_\xi^{-1}(\delta).
		\]
	\end{definition}
	
	Returning to our example one final time: the hyperplane $\pi_\xi^{-1}(\delta)$ is a non-horizontal line in the plane $A$ which passes near, but not through, the central point $p$. In a small $\epsilon$-ball around $p$, this line generically intersects the arcs which form $\bN_\epsilon(p)$ in two points, so the complex link $\Clk_X$ in this case is just the two-point space:
	
	\begin{figure}[h!]
		\includegraphics[scale=.4]{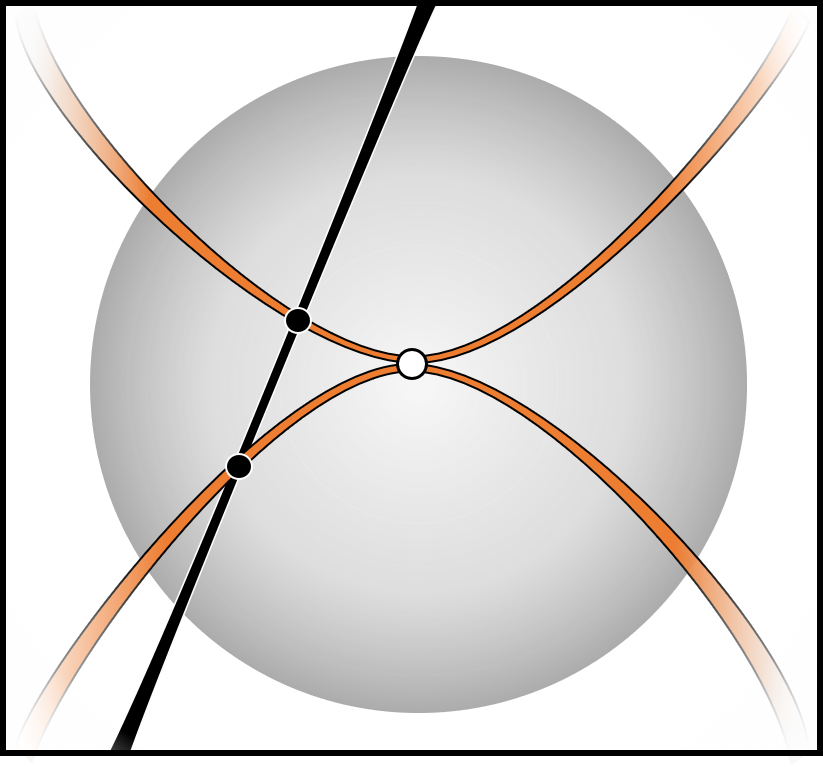}
	\end{figure}
	
	\noindent The (stratified homeomorphism type of the) complex link $\Clk_X$ depends only on the stratum $X$, and not on the auxiliary choices $(p,A,\epsilon,\xi,\delta)$ described above \cite[Chapter II.2]{SMTbook}. It is also interesting to note that the invariance of $\Clk_X$ to the chosen direction $\xi$ is entirely a feature of complex analytic geometry --- for real analytic Whitney stratified spaces, the intersection $\bN_\epsilon(p) \cap \pi_\xi^{-1}(\delta)$ is liable to change as $\xi$ is varied.
	
	\begin{remark}
		In this paper we will be exclusively interested in the complex link of a point (i.e., a zero-dimensional stratum) within a complex projective variety. In this  special case, one is not required to construct a normal slice, so the formula from Definition \ref{def:clink} reduces to $\Clk_X = \bW \cap \BB_\epsilon(p) \cap \pi_\xi^{-1}(\delta)$. 
	\end{remark}

	\section{The Hilbert-Samuel Multiplicity}\label{sec:algmult}

	Let $R = \CC[x_0,\ldots,x_n]$ denote either the coordinate ring of projective space $\pp^n$ or affine space $\CC^{n+1}$. In the projective case we will implicitly assume that $R$ is graded and all of its ideals considered below are homogeneous. Let $X$ be an irreducible complex (affine or projective) algebraic variety given by a prime ideal $I \lhd R$ and let $Y$ be a scheme corresponding to a primary ideal $J \subset I$. The {\em local ring of $Y$ along $X$}, usually written $\cO_{X,Y}$, is the localization of $(R/J)$ at $I$. The following notion is due to Samuel \cite{samuel1955methodes}.
	
	\begin{definition} \label{def:algmult}
		Let $\cM$ be the maximal ideal of $\cO_{X,Y}$ and let $c$ be the codimension $\dim Y - \dim X$. The {\bf Hilbert-Samuel function} of $Y$ along $X$ is \[
		\textbf{HS}(t) = \text{length}\left(\cO_{X,Y}/\cM^t\right).
		\]
		For all $t \gg 0$, this function is a polynomial in $t$ of degree $c$ whose leading coefficient is a strictly positive integer divisible by $c!$ --- and the {\bf Hilbert-Samuel multiplicity of $Y$ along $X$}, written $\e_XY$, is the leading coefficient of the normalized polynomial $(1/c!)\cdot {\bf HS}(t)$.
	\end{definition}
	
	It is shown in \cite[Section 4.3]{fulton2013intersection} that $\e_XY$ is also equal to the coefficient of $[X]$ in the {\bf Segre class} $s(X,Y)$, which naturally lives in the {\em Chow group} of $X$ (or in the {\em Chow ring} of an ambient smooth variety $M$ via push-forward, $X\subset Y\subset M$; we will often work in the $M=\pp^n$ setting). 
	
	\subsection{Multiplicities from Degrees}
	
	When $X \subset Y \subset \pp^n$ are projective varieties, it is often algorithmically convenient to extract $\e_XY$ from a choice $B_X = \set{f_0,\dots, f_r}$ of homogeneous polynomials that generate the defining ideal $I \lhd R$ of $X$. We will assume here that all the $f_i$ have the same degree $d$, which is always possible to arrange without loss of generality \cite[Section 2.1.4]{HH19}. Let $L_i$ be a generic $(n-i)$-dimensional linear subspace of $\pp^n$, and let $V_i \subset \pp^n$ be the varieties given by 
	\begin{align} \label{eq:V}
	V_i = \set{x \in \pp^n \mid F_1(x) = F_2(x) = \cdots = F_{\dim Y-i}(x) = 0},
	\end{align} where the $F_j$ are homogeneous polynomials of degree $d$ that have the form
	\[
	F_j = \sum_{k=0}^r \lambda_k^j f_k
	\]
	for general choices of $\lambda_k^j \in \CC$. (Note $V_i$ contains $X$ for all $i$ by design). 
	
	In \cite{HH19} it is shown that the Segre class $s(X,Y)$, and hence the  multiplicity $\e_XY$, is determined by the numbers
	\begin{align}\label{eq:lambdadeg}
	\LC^i_XY = \deg(Y) \cdot d^{\dim Y - i}- \deg((Y \cap V_i \cap L_i) - X),
	\end{align}
	for each $i$ between $0$ and $\dim X$. In particular, \cite[Theorem 5.3]{HH19} establishes that 
	\begin{align}\label{eq:multlambda}
	\e_XY = \frac{\LC^{\dim X}_XY}{\deg X},
	\end{align} 
	or more explicitly,
	\begin{align}\label{eq:multlambdalong}
	\e_XY = \frac{\deg(Y) \cdot d^{\dim Y - \dim X}- \deg((Y \cap V_{\dim X} \cap L_{\dim X}) - X)}{\deg X}.
	\end{align}
	In order to treat affine varieties on an equal footing with projective ones when it comes to using (\ref{eq:multlambda}) and related formulas, we will appeal to the following result. The idea is to replace the affine varieties $X\subset Y$ in $\CC^n$ by their {\em projective closures} $PX \subset PY$ in $\pp^n$ --- see \cite[Exercise I.2.9]{Hartshorne1974} for a definition.  
	\begin{proposition}\label{prop:affine_mult}
		Let $X\subset Y$ be closed subvarieties of affine space $\CC^n$ and let $PX \subset PY$ denote their projective closures in $\pp^n$. Then, we have
		\[
		\e_XY=\frac{\LC_{PX}^{\dim PX} PY}{\deg PX},
		\] 
		i.e., the multiplicity of $Y$ along $X$ can be computed from (\ref{eq:multlambda}) applied to the projective closures of $X$ and $Y$.
	\end{proposition}
	\begin{proof}
		Select an affine chart of $\pp^n$ specified by a general hyperplane $H$, i.e. $\CC^n\cong \pp^n-H$ and an open set $U \subset \pp^n$ so that ({\bf a}) $PY \cap U$ is open and dense in $PY$,  ({\bf b}) if $V=U-H$ then $Y \cap V$ is open and dense in $Y\subset \CC^n$, and ({\bf c}) $PX \cap U$ is open and dense in $PX$ (i.e. $PX$ is not contained in $\pp^n-U$). By ({\bf a}) and ({\bf c}) we obtain an isomorphism of local rings \[\cO_{PX,PY} \simeq \cO_{PX \cap U,PY \cap U}.\] Next, since $X$ is a subvariety of $Y$ with $Y \cap V$ dense in $Y$ by ({\bf b}), we have a second isomorphism of local rings \[
		\cO_{PX \cap U,PY \cap U} \simeq \cO_{X \cap V,Y \cap V}.\] And finally, since $Y\cap V$ is dense in $Y$ and since $PX$ is not contained in $\pp^n-U$, then $X\cap V$ is dense in $X$ and we have a third isomorphism \[\cO_{{X}\cap V,{Y}\cap V} \simeq \cO_{X,Y}.\] Stringing together these three isomorphims, one obtains $\cO_{X,Y} \simeq \cO_{PX,PY}$; since the multiplicities $\e_XY$ and $\e_{PX}{PY}$ are completely determined by the corresponding local rings, they must be equal. The desired conclusion now follows from (\ref{eq:multlambda}).  
	\end{proof}
	
	We expect that some version of the above argument, (i.e., that $\e_XY$ must equal $\e_{PX}{PY}$ because the two associated local rings are isomorphic) already exists in the literature, but we were unable to locate it and have therefore included this proof for completeness. This result facilitates the use of \eqref{eq:multlambdalong} for a pair of affine varieties.
	
	\subsection{Multiplicities of Linear Sections}\label{sec:eslice}
	
	Here we describe the behaviour of the Hilbert-Samuel multiplicity $\e_XY$ for complex projective varieties $X \subset Y$ when both $X$ and $Y$ are replaced by their intersections with (sufficiently generic) linear spaces. The proposition below can be seen as a direct consequence of standard properties of Segre classes (along with the relation between Segre classes and multiplicities).
	
	\begin{proposition}\label{prop:algmultsection}
		Let $Y$ be a pure dimensional subscheme of the complex projective space $\PP^n$, let $X$ be an irreducible subvariety of $Y$ and let $L \subset \pp^n$ be given by an intersection
		\[
		L = H_1 \cap H_2 \cap \cdots \cap H_\ell,
		\]
		where each $H_i \subset \pp^n$ is a generic hyperplane. If the codimension $\ell = n - \dim L$ is strictly less than $\dim X$, then the multiplicities $\e_XY$ and $\e_{X \cap L}(Y \cap L)$ are equal. Further if $\ell=\dim(X)$ then $\e_XY=\e_p(Y\cap L)$, where $p$ is any of the $\deg(X)$ points in $X\cap L$. 
	\end{proposition} 
	\begin{proof}In this proof we will work with the pushforward to the Chow ring of $\pp^n$ of the Segre class $s(X,Y)$, in a slight abuse of notation this will also be denoted $s(X,Y)$. Denote this Chow ring as $A^*(\PP^n)\cong \ZZ[h]/\langle h^{n+1} \rangle$ where $h$ is the rational equivalence class of a general hyperplane. 
		Since each $H_i$ is a general divisor on $\pp^n$, the coefficient of $h^{\dim X - \ell}$ in the Segre class $s(X \cap L,Y \cap L)$ equals the coefficient of $h^{\dim X}$ in the Segre class $s(X,Y)$, i.e.,
		\[
		\{s(X\cap L,Y\cap L)\}_{\dim X - \ell}  = \{s(X,Y)\}_{\dim X} \cdot h^\ell.
		\]
		A proof of the above property of Segre classes can be found, for example, in \cite[Corollary~3.2]{Harris2017}. First suppose that $\ell<\dim(X)$. Using the fact that $\e_{X \cap L}(Y \cap L)$ is the coefficient of $[X \cap L]$ in $s(X \cap L,Y \cap L)$, one obtains 
		\begin{align*}
		\{s(X\cap L,Y\cap L)\}_{\dim X - \ell} & = \e_{X\cap L}(Y\cap L) \cdot [X\cap L] \\
		&=\e_{X\cap L}(Y\cap L) \cdot \deg X \cdot h^{n-\dim X+\ell},
		\end{align*}
		where the second equality follows from the fact that each $H_i$ is a general divisor, so in particular $\deg X = \deg (X \cap L)$. Now take $\ell=\dim(X)$. Then $X\cap L$ consists of $\deg(X)$ reduced points $p_1,\dots, p_{\deg(X)}$; by \cite[Example 4.3.4]{fulton2013intersection} we have that \begin{align*}
		\{s(X\cap L,Y\cap L)\}_{\dim X - \ell} & = \e_{p_1}(Y\cap L)[p_1] +\cdots +\e_{p_{\deg(X)}}(Y\cap L)[p_{\dim(X)}]\\
		&=\e_p(Y\cap L)\deg(X)[p]\\
		&=\e_p(Y\cap L) \cdot \deg X \cdot h^{n-\dim X+\ell}
		\end{align*} where $p$ is any point in $X\cap L$ (all of which are rationally equivalent). 	On the other hand, we also have
		\begin{align*}
		\{s(X,Y)\}_{\dim X} \cdot h^\ell  &=\e_XY\cdot [X] \cdot h^\ell \\
		&= \e_XY\cdot \deg X \cdot h^{n-\dim X+\ell},
		\end{align*} which forces $\e_{X\cap L}(Y\cap L)=\e_XY$ for $\ell<\dim(X)$ and $\e_{p}(Y\cap L)=\e_XY$ when $\ell=\dim(X)$ as desired.
	\end{proof}
	Having cut $X$ down to a single point $p$, we turn our attention to simplifying $Y$. The next result uses \eqref{eq:multlambda} to show that we can always replace $Y$ by a curve, i.e., a one-dimensional projective variety, when computing $\e_pY$.
	
	\begin{proposition}\label{prop:algmultcurve}
		Let $Y$ be a pure dimensional subscheme of the complex projective space $\PP^n$, let $p$ be any reduced point in $Y$, and consider a linear space $L \subset \pp^n$ given by the intersection of $m \geq 0$ general hyperplanes containing $p$. If $m \leq  \dim Y - 1$, then $\e_p(Y \cap L)$ is well-defined and equals $\e_{p}Y$.
	\end{proposition}
	\begin{proof}
		The generating ideal of $p$ in the polynomial ring $\CC[x_0,\dots,x_n]$ can be chosen to consist of $n$ linear forms $\{\ell_1,\ldots,\ell_n\}$. It follows that $d=1$ and $\deg X=1$ in \eqref{eq:multlambdalong}. Since $\dim p = 0$, we have
		\[
		\e_{p}Y=\deg(Y)-\deg((Y\cap V_{\dim Y})-p), 
		\] where $V_{\dim Y}$ is the variety defined by the polynomials $\{P_1,\ldots,P_{\dim Y}\}$, with each $P_j$ being a linear combination of the form \[P_j = \sum_{i=1}^n \lambda^j_i \ell_i \text{ for general } \lambda^j_i\in \CC.\] Without loss of generality, we may take $L$ to be the variety defined by the first $m$ of these, say $\{P_1,\dots, P_m\}$. Thus, $L$ is a linear system with base locus $p$, so it forms a smooth complete intersection outside of $p$. Hence the intersection $Y\cap L$ is transverse in the expected dimension, i.e.~in dimension $\dim (Y\cap L) = \dim Y - m > 0$ and moreover, $\deg(Y\cap L)=\deg(Y)$. Letting $V_{> m}$ be the variety defined by $\{P_{m+1},\ldots P_{\dim Y}\}$, we have
		\begin{align*}
		\e_{p}Y&=\deg(Y)-\deg((Y\cap V_{\dim Y})-p) \\&=\deg(Y\cap L)-\deg((Y\cap L \cap V_{> m})-p)  \\
		&=\e_{p}(Y\cap L).
		\end{align*} This argument fails whenever $m=\dim Y$, since the intersection $Y\cap L$ may not be transverse in this case. 
	\end{proof}
	
	The following theorem serves to summarize the main results in this section by combining Propositions \ref{prop:algmultsection}  and \ref{prop:algmultcurve}. 
	
	\begin{theorem}\label{thm:algmultslice}
		Let $X \subset Y$ be a pair of complex projective subvarieties of $\pp^n$, and let $L \subset \pp^n$ be a linear space given by the intersection of $k \geq 0$ general hyperplanes $H_1, \ldots, H_k$ containing a point $p$ of $X$.
		\begin{enumerate}
			\item If $k \leq \dim X$, then $\e_XY = \e_{X \cap L}(Y \cap L)$.
			\item If $k = \dim X$, then $\e_p(Y \cap L) = \e_XY$.
			\item If $\dim X < k \leq \dim Y - 1$, then $\e_XY = \e_p(Y \cap L)$.
		\end{enumerate}
	\end{theorem}
	
	\noindent This is Theorem ({\bf A}) from the Introduction. We note that we may allow $Y$ in the statement above to be any pure dimensional subscheme of $\pp^n$, but have restricted the case where $Y$ is a variety as this will be the only case we employ in later sections. Assertion (3) of this result (for $k = \dim Y - 1$) implies that the evaluation of $\e_XY$ for arbitrary projective varieties $X \subset Y$ in $\pp^n$ can be reduced to the computation of $\e_pC$ where $p$ is a point lying on the curve $C = Y \cap L$; this scenario will be the central focus of the next section.
	
	\section{Point-Curve Multiplicities via Complex Links}\label{sec:pointcurve}
	
	Our goal here is to provide a stratified Morse-theoretic proof of  Theorem ({\bf B}) from the Introduction. 
	\begin{theorem}\label{thm:toplink}
		If $p$ is any (possibly singular) point on a curve $C \subset \pp^n$, then we have
		\[
		\e_pC = \chi(\Clk_p),
		\]
		where $\e_pC$ is the Hilbert-Samuel multiplicity (from Definition \ref{def:algmult}) and $\chi(\Clk_p)$ is the Euler characteristic of $p$'s complex link in $C$ (from Definition \ref{def:clink}).
	\end{theorem}
	Since $p$ can be defined as the zero set of $n$ linear polynomials, by \eqref{eq:lambdadeg} we have 
	\[
	\e_pC = \LC^0_pC = \deg(C) - \deg((C \cap H_p) - p), 
	\] where $H_p$ is a generic hyperplane in $\pp^n$ passing through $p$. Let $\xi$ be the unit normal to $H_p$, and denote by $\pi_\xi:\pp^n \to \RR$ the projection map $z \mapsto \text{Re}\ip{z-p,\xi}$ so that the level set of $\pi_\xi$ at $0$ is precisely $H_p$. By our genericity assumption on $H_p$, the vector $\xi$ is nondegenerate in the sense of Definition \ref{def:xi}. Thus, there is some small positive $\delta$ so that $\pi_\xi$ is a stratified Morse function on $C$ with no critical point other than $p$ taking values in $[-\delta,\delta]$. In particular, this means that no singular points of $C$ other than $p$ are allowed to lie in $C \cap \pi_\xi^{-1}[0,\delta]$, which is the part of $C$ lying within the shaded rectangular region below\footnote{Here we have resorted to drawing a real picture for simplicity; the one-dimensional complex curve $C$ should in fact have $\RR$-dimension two.}.
	
	\begin{center}
		\includegraphics[scale=.55]{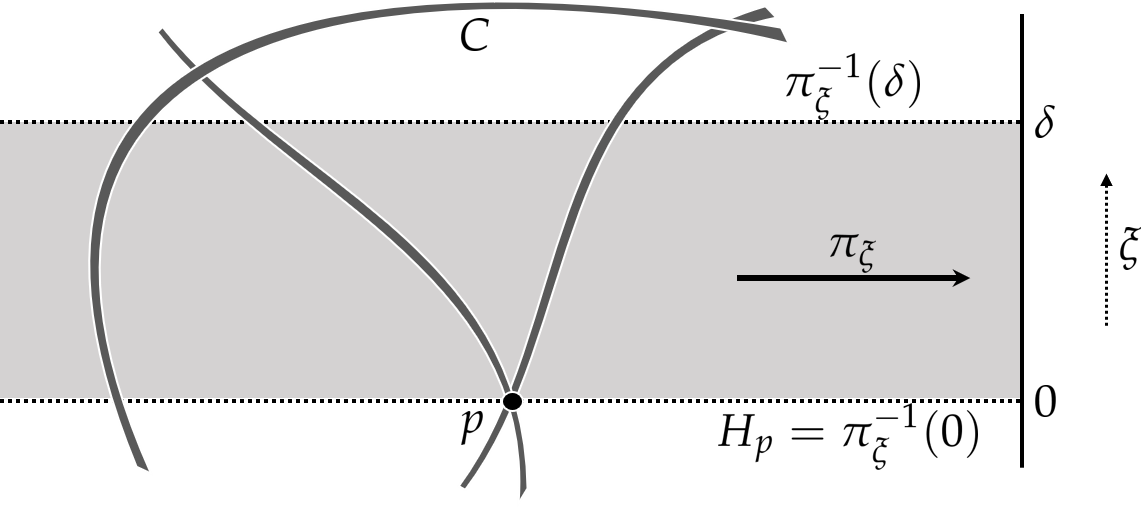}
	\end{center}
	
	The upper levelset $\pi_\xi^{-1}(\delta)$ intersects $C$ in a set of cardinality $\deg C$ since $\pi_\xi^{-1}(\delta)$ is sufficiently generic. On the other hand, since the lower levelset $\pi_\xi^{-1}(0)$ is forced to pass through $p$, it may intersect $C$ in fewer points. Thus, the quantity of interest to us here is
	\begin{align}
	\LC^0_pC = \#\set{C \cap \pi_\xi^{-1}(\delta)} - \#\set{C \cap \pi_\xi^{-1}(0)} + 1, \label{eq:lambda0}
	\end{align}
	where the last $+1$ term comes from the fact that we are required to discard $p$ from the second intersection. The main tool in our argument here is one of Thom's celebrated Isotopy Lemmas --- see \cite[Proposition 11.1]{Mather2012} or \cite[Chapter I.1.5]{SMTbook}. 
	
	\begin{lemma}\label{lem:thom}{\bf [Thom's first isotopy lemma]}
		Let $M$ and $N$ be smooth manifolds and $Z \subset M$ a Whitney stratified subset. If $f:M \to N$ is a smooth proper map whose restriction $f|_X$ to each stratum $X \subset Z$ is a submersion (i.e., the derivative $df_p:T_pX \to T_{f(p)}N$ is surjective for all $p$ in $X$), then $f|_X:X \to f(X)$ is a (locally trivial) fiber bundle.
	\end{lemma}
	
	By our choice of $\delta$, the function $\pi_\xi$ when restricted to $C \cap \pi_\xi^{-1}(0,\delta)$ satisfies the hypotheses of this lemma. Reducing $\delta$ further if necessary, we are therefore guaranteed the existence of a local trivialization, i.e., a homeomorphism 
	\[
	\left[C \cap \pi^{-1}_\xi(0,\delta)\right] \simeq (0,\delta) \times \left[C \cap \pi_\xi^{-1}(\delta)\right].
	\] Our strategy here is to examine the following zigzag diagram of inclusion maps
	\begin{align}
	C \cap \pi_\xi^{-1}(0) \hookrightarrow C \cap \pi_\xi^{-1}[0,\delta] \hookleftarrow C \cap  \pi_\xi^{-1}(\delta).\label{eq:zz}
	\end{align} This next result is concerned with the first inclusion. 
	
	\begin{proposition}\label{prop:zerodelta}
		The inclusion $C \cap \pi_\xi^{-1}(0) \hookrightarrow C \cap \pi_\xi^{-1}[0,\delta]$ is a homotopy equivalence, and in particular it admits a homotopy-inverse $\phi:C \cap \pi_\xi^{-1}[0,\delta] \to C \cap \pi_\xi^{-1}(0)$.
	\end{proposition}
	\begin{proof}
		By Thom's first Isotopy Lemma applied to the restriction of $\pi_\xi$ to $C \cap \pi_\xi^{-1}(0,\delta]$, the constant vector field $-\xi$ on $(0,\delta]$ lifts to a vector field $V$ on $C \cap \pi_\xi^{-1}(0,\delta]$ so that the differential $d\pi_\xi$ sends each vector of $V$ to $-\xi$, as depicted below:
		
		\begin{center}
			\includegraphics[scale=.55]{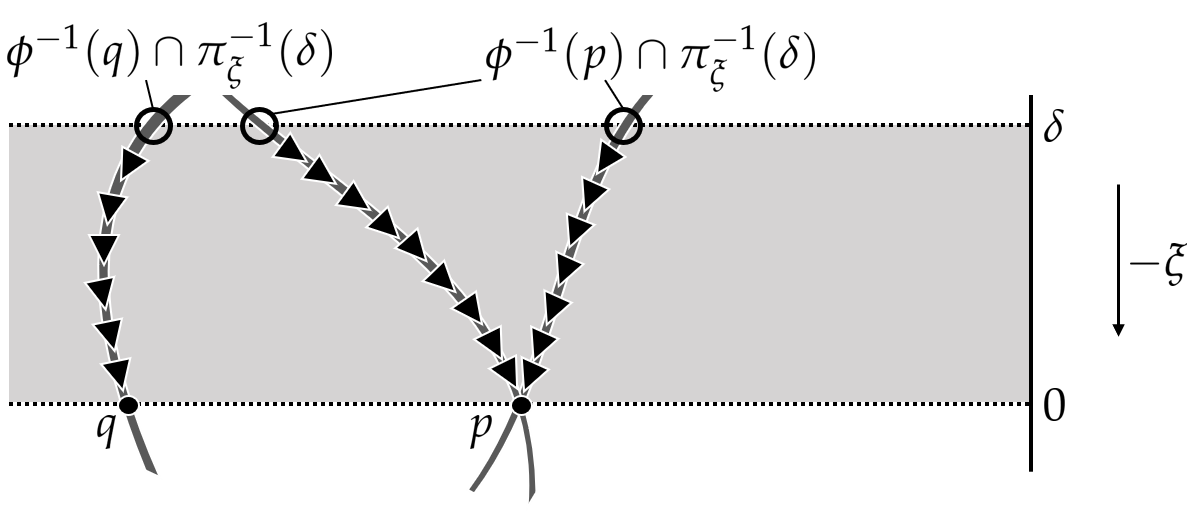}
		\end{center}
		\noindent The desired map $\phi$ is obtained by flowing along the integral curves of $V$.
	\end{proof}
	
	The second map from our zigzag (\ref{eq:zz}) will be described via the corresponding relative homology group, namely
	\[
	\HG_\bullet\left(C \cap \pi_\xi^{-1}[0,\delta],C \cap \pi_\xi^{-1}(\delta)\right),
	\] where we have implicitly assumed rational coefficients throughout. The following result shows that this group only depends on local data pertaining to the fibers of $\phi$ over $p$.
	
	\begin{lemma}\label{lem:local}
		Let $\phi:C \cap \pi_\xi^{-1}[0,\delta] \to C \cap \pi_\xi^{-1}(0)$ be a homotopy inverse to the inclusion (as from Proposition \ref{prop:zerodelta}). There is an isomorphism of relative homology groups:
		\[
		\HG_\bullet\left(C \cap \pi_\xi^{-1}[0,\delta],C \cap \pi_\xi^{-1}(\delta)\right) \simeq  \HG_\bullet\left(\phi^{-1}(p),\phi^{-1}(p)\cap\pi_\xi^{-1}(\delta)\right).
		\]
		Consequently, the associated Euler characteristics satisfy
		\[
		\chi\left(C \cap \pi_\xi^{-1}[0,\delta]\right) - \chi\left(C \cap \pi_\xi^{-1}(\delta)\right) = 1-\chi\left(\phi^{-1}(p) \cap \pi_\xi^{-1}(\delta)\right).
		\]
	\end{lemma}
	\begin{proof}
		The set $C \cap \pi_\xi^{-1}[0,\delta]$ decomposes as a disjoint union
		\[
		C \cap \pi_\xi^{-1}[0,\delta] = \coprod_q \phi^{-1}(q),
		\]
		where $q$ ranges over the points in $C \cap \pi_\xi^{-1}(0)$. By the additivity of homology, we have a direct sum decomposition
		\[
		\HG_\bullet\left(C \cap \pi_\xi^{-1}[0,\delta],C \cap \pi_\xi^{-1}(\delta)\right) = \bigoplus_q \HG_\bullet\left(\phi^{-1}(q),\phi^{-1}(q)\cap\pi_\xi^{-1}(\delta)\right).
		\]
		It therefore suffices to show that the summands corresponding to $q \neq p$ are all trivial. Since no such $q$ is a critical point of $\pi_\xi$, the vector field $V$ on $C \cap \pi_\xi^{-1}(0,\delta]$ which was used to construct $\phi$ in the proof of Proposition \ref{prop:zerodelta} extends non-trivially through $q$. Now by Thom's Isotopy Lemma \ref{lem:thom} above, the stratified homeomorphism type of $\phi^{-1}(q) \cap \pi_\xi^{-1}(t)$ remains unchanged across all $t \in [0,\delta]$, so in particular there is a homeomorphism of pairs 
		\[
		\Big(\phi^{-1}(q), \phi^{-1}(q) \cap \pi_\xi^{-1}(\delta)\Big) \simeq \Big([0,\delta],\delta\Big),
		\] and hence the relative homology is trivial as desired. To extract the statement about the Euler characteristics from the statement about relative homology groups, one uses the observation that $\phi^{-1}(p)$ is homeomorphic to the cone at $p$ over $\phi^{-1}(p) \cap \pi_\xi^{-1}(\delta)$. Since all cones are contractible, we obtain $\chi(\phi^{-1}(p)) = 1$.
	\end{proof}
	
	To conclude our proof of Theorem \ref{thm:toplink}, we observe that
	\begin{align*}
	\e_pC &=  \#\set{C \cap \pi_\xi^{-1}(\delta)} - \#\set{C \cap \pi_\xi^{-1}(0)} + 1, & \text{by }(\ref{eq:lambda0}) \\
	&= \chi\left(C \cap \pi_\xi^{-1}(\delta)\right) - \chi\left(C \cap \pi_\xi^{-1}(0)\right) + 1 & \text{since $\dim C = 1$} \\
	&= \chi\left(C \cap \pi_\xi^{-1}(\delta)\right) - \chi\left(C \cap \pi_\xi^{-1}[0,\delta]\right) + 1 & \text{by Proposition } \ref{prop:zerodelta} \\
	&= \chi\left(\phi^{-1}(p) \cap \pi_\xi^{-1}(\delta)\right) & \text{by Lemma } \ref{lem:local} \\
	&= \chi(\Clk_pC) & \text{by Definition }\ref{def:clink}
	\end{align*} 
	as desired.
	
	\begin{remark} \label{rem:topargcodim}
		Neither Proposition \ref{prop:zerodelta} nor Lemma \ref{lem:local} require any constraint on the dimension of $C$, and both would work just as well when $\dim C > 1$. On the other hand, it is only when $\dim C = 1$ that one obtains $\dim (C \cap \pi_\xi^{-1}(t)) = 0$ for $t$ in $[0,\delta]$, and it is a miracle of zero-dimensionality that degree and Euler characteristic coincide. This accident is exploited only once in our argument, i.e., when transitioning from the first line to the second one in the string of equalities above. 
	\end{remark}
	
	Unfortunately, we do not anticipate any direct relationship between degrees and Euler characteristics of higher-dimensional projective  varieties. Thus, this argument does not extend directly to the scenario where our curve $C$ is replaced by a variety $Y$ of dimension $> 1$. In any event, Theorems \ref{thm:algmultslice} and \ref{thm:toplink} guarantee that all Hilbert-Samuel multiplicity computations can be reduced to Euler characteristic estimation for a finite collection of points. We now turn our attention to inferring such multiplicities from point samples. 
	
	\section{Estimating Multiplicities from Finite Samples}
	\label{section:Multiplcity_From_Finite_Samples}
	The {\em reach} $\tau_M > 0$ of a compact submanifold $M \subset \RR^n$ is the smallest radius $r > 0$ for which the radius-$r$ normal bundle around $M$ self-intersects. This notion was first introduced by Federer in \cite{federer}, and it serves as an important measure of the regularity of the embedding $M \hookrightarrow \RR^n$. The reciprocal $1/\tau_M$, called the {\em condition number} of $M$, features prominently in the homological inference results of Niyogi, Smale and Weinberger from \cite{niyogietal}. These results have been extended to the case where $M$ has a smooth boundary $\partial M$ by Wang and Wang \cite{wang2}. In this setting, the role of the reach is played by a new parameter
	\begin{align}\label{eq:Dwang}
	\Delta_M := \min\set{\tau_M, \tau_{\partial M}, \rho_M},
	\end{align}
	where $\rho_M$ is the largest radius $r > 0$ so that at each point $x$ in $X$, the exponential map $M \to T_xM$ is a diffeomorphism onto its image when restricted to the open ball $\BB_r(x)^\circ \cap M$. The following result is \cite[Theorem 3.3]{wang2}.
	
	\begin{theorem} \label{thm:wangdense}
		Let $M$ be a smooth, nonempty $k$-dimensional submanifold with boundary of $\RR^n$. For any radius $r \in (0,\Delta_M/2)$ and probability $\gamma \in (0,1)$, there exists an explicit bound $N_M(r,\gamma)$ satisfying the following property. Any uniformly sampled\footnote{i.e., independent and identically distributed with respect to the uniform measure on $M$.} finite set $S \subset M$ of cardinality larger than $N_M(r,\gamma)$ is $(r/2)$-dense in $M$ with probability exceeding $(1-\gamma)$.
	\end{theorem}
	
	\noindent In other words, we can guarantee with high confidence that every point of $M$ is no more than $r/2$ away from some point of $S$ whenever $\#S > N_M(r,\gamma)$. 
	
	\begin{remark} This bound $N_M(r,\delta)$ has the form 
		\begin{align}\label{eq:N}
		N_M(r,\gamma) = \beta_M(r) \cdot \left[\beta_M\left(\frac{r}{2}\right) + \ln\left(\frac{1}{\gamma}\right)\right],
		\end{align}
		where $\beta_M:\RR \to \RR_{> 0}$ is the function
		\[
		\beta_M(x) := \frac{\text{Vol}(M)}{\frac{\cos^k(\theta)}{2^{(k+1)}}\cdot I_{y}\left(\frac{k+1}{2},\frac{1}{2}\right) \cdot \text{Vol}(\BB_x^k)}.
		\]
		Here $\text{Vol}(\bullet)$ is standard $k$-dimensional Lebesgue volume and the auxiliary variables are 
		\[
		\theta := \arcsin\left(\frac{x}{4\Delta_M}\right) \quad \text{ and } \quad y := 1-\frac{x^2 \cdot \cos^2(\theta)}{16\Delta_M^2}; 
		\] moreover, $I_y(a,b)$ denotes the {\em regularized incomplete beta function}
		\begin{align*}
		I_y(a,b) := \frac{B_y(a,b)}{B_1(a,b)}, \quad \text{ with } \quad 
		B_y(a,b) := \int_0^y t^{a-1}(1-t)^{b-1}~dt.
		\end{align*}
		And finally, $\BB^k_x$ is the ball of radius $x$ in $k$-dimensional Euclidean space.
	\end{remark}
	
	\subsection{Setup and Parameter Choices}
	
	Let $p$ be any (not necessarily singular) point on a curve $C \subset \PP^n$. Passing to an affine chart of $\pp^n$ containing $p$, we may as well work within $\CC^{n+1} \simeq \RR^{2n+2}$. In light of this identification, all dimensions of spaces mentioned henceforth are to be understood as dimensions over $\RR$ rather than $\CC$. Consider any choice of positive radii $\epsilon \gg \delta$ and (unit length) direction $\xi$ so that the complex link of $p$ in $C$ is given by the intersection
	\begin{align}\label{eq:embclk}
	\Clk_p = C \cap \BB_\epsilon(p) \cap \pi_\xi^{-1}(\delta).
	\end{align}
	By Theorem \ref{thm:toplink}, this is a set of (finite) cardinality $\ell := \e_pC$, so we may enumerate its points as $\set{x_1,x_2,\ldots,x_\ell}$. Let $\mu > 0$ be the smallest pairwise Euclidean distance between these points
	\begin{align}\label{eq:mu}
	\mu := \min\set{\|x_i-x_j\|, \text{ where } 1 \leq i \neq j \leq \ell},
	\end{align} and let $\kappa > 0$ be the distance between $\Clk_p$ and the boundary of the closed ball $\BB_\epsilon(p)$:
	\begin{align}\label{eq:kappa}
	\kappa := \min\set{\epsilon - \|p-x_i\|, \text{ where } 1 \leq i \leq \ell}.
	\end{align}
	These new distances $\kappa$ and $\mu$ are determined by the initial choices of $\epsilon, \delta$ and $\xi$. We also select a new parameter $\epsilon_0 \in (0,\delta)$, called the {\em inner radius}. Since $\epsilon_0$ is smaller than $\delta$, the open ball $\BB_{\epsilon_0}(p)^\circ$ does not intersect the offset hyperplane $\pi_\xi^{-1}(\delta)$, and hence does not contain any of the points $\set{x_1,\ldots,x_\ell}$. 
	
	\begin{proposition}\label{prop:annmwb}
		The intersection $C' := C \cap \left[\BB_\epsilon(p) - \BB_{\epsilon_0}(p)^\circ\right]$ forms a two-dimensional manifold with boundary embedded within $C_\text{\rm reg} \subset \RR^{2n+2}$.
	\end{proposition}
	\begin{proof}
		We recall that the radius $\epsilon$ from \eqref{eq:embclk} satisfies the property that the boundary sphere $\partial\BB_e(p)$ is transverse to $C_\text{reg}$ for all $e \in (0,\epsilon]$. Thus, the smooth map ${\bf d}_p:\RR^{2n+2} \to \RR$ given by 
		\[
		x \mapsto \|x-p\|^2
		\] has no critical points on $C_\text{reg}$ valued in $(0,\epsilon]$. By Lemma \ref{lem:thom}, the restriction ${\bf d}_p:C' \to [\epsilon_0,\epsilon]$ forms a trivial fiber bundle. Since all such $e$ are regular values of ${\bf d}$, the implicit function theorem guarantees that each fiber $F_e := C_\text{reg} \cap {\bf d}_p^{-1}(e)$ is a smooth one-dimensional submanifold of $C'$.  Therefore, the desired result now follows from the fact that $C'$ is diffeomorphic to the product $F \times [\epsilon_0,\epsilon]$ where $F$ is a smooth one-dimensional manifold. 
	\end{proof}
	
	Recalling the fact that $\set{x_1,\ldots,x_\ell} \cap \BB_{\epsilon_0}(p)$ is empty since $\epsilon_0 < \delta$, we have
	\[
	C' \cap \pi_\xi^{-1}(\delta) = \set{x_1,\ldots,x_\ell}.
	\]
	An immediate side-effect of replacing $C'$ by a dense uniform sample is that none of the sample points will lie exactly on $\pi_\xi^{-1}(\delta)$. Therefore, we require a final {\em thickness} parameter $\alpha > 0$. This is a sufficiently small positive radius for which the following property holds: the set of all points in $C$ that lie within distance $\alpha$ of the offset hyperplane $\pi_\xi^{-1}(\delta)$ is entirely contained within the union of open radius-$\alpha$ balls around points of the complex link:
	\begin{align}\label{eq:alphasep}
	\set{y \in C \mid \text{dist}[y,\pi_\xi^{-1}(\delta)] < \alpha} \subset \bigcup_{i=1}^\ell \BB_\alpha(x_i).
	\end{align}
	By sufficiently small here we mean that $\alpha$ must be bounded above by the minimum
	\begin{align}\label{eq:alphabound}
	\alpha < \min\set{(\epsilon-\delta), (\delta-\epsilon_0),{\mu}/{4},\kappa,{\Delta_{C'}}/{2}}.
	\end{align}
	This inequality encodes all of the geometric constraints required for our inference result. As atonement for introducing this deluge of parameters, we remind the reader that $\epsilon$ and $\delta$ were fixed in \eqref{eq:embclk}, while $\epsilon_0$ is the inner radius used in Proposition \ref{prop:annmwb}; the quantities $\mu$ and $\kappa$ are described in \eqref{eq:mu} and \eqref{eq:kappa} respectively, and $\Delta_{C'}$ is from \eqref{eq:Dwang}. Below we have illustrated the typical local picture of $C$ near $p$ in the case where $\e_pC = 3$.
	
	\begin{center}
		\includegraphics[scale=.4]{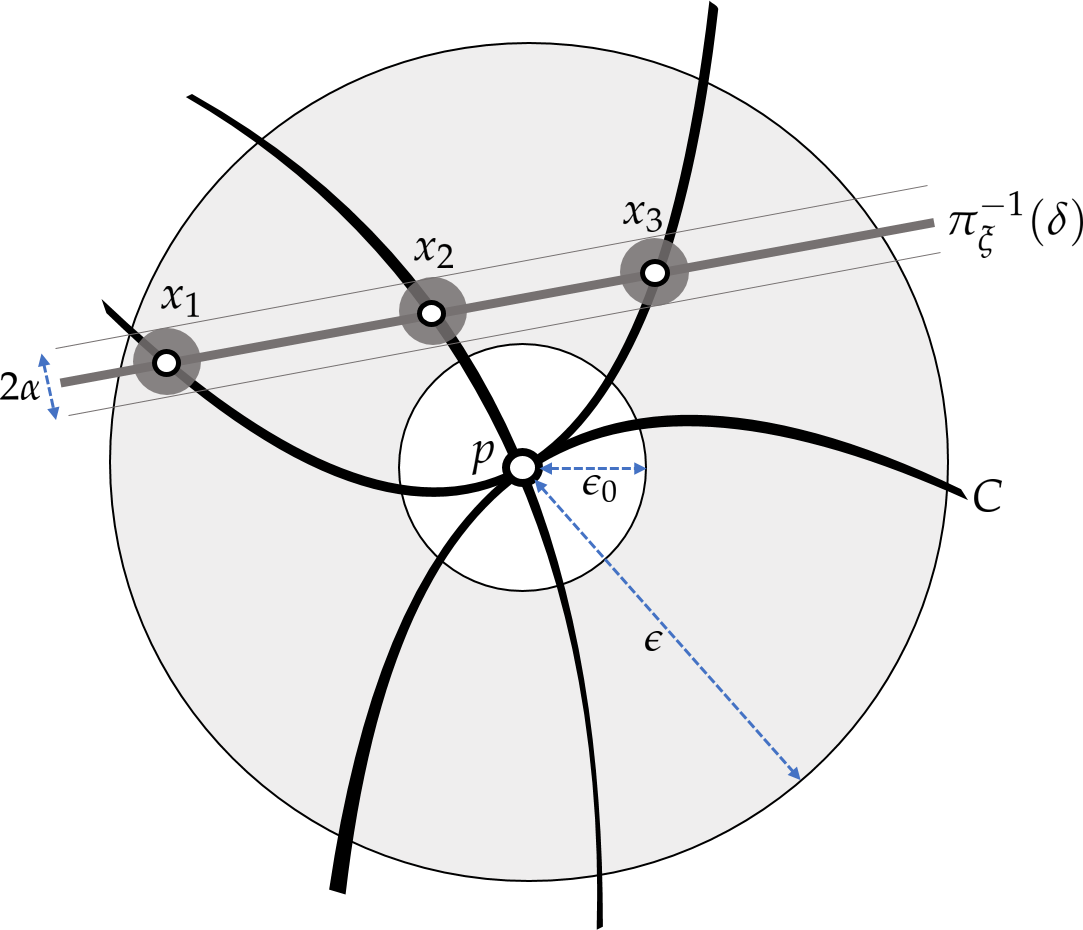}
	\end{center}
	
	\noindent The first four terms within the minimum on the right side of \eqref{eq:alphabound} are designed simply to ensure that balls of radius $\alpha$ around each of the $x_i$ are well-separated and fully contained within the annulus $[\BB_{\epsilon}(p) - \BB_{\epsilon_0}(p)^\circ]$; the final term is required for applying Theorem \ref{thm:wangdense}.
	
	\subsection{Inferring Multiplicities with High Confidence}	
	
	Here we prove Theorem {({\bf C}) from the Introduction. The parameters encountered in its statement below were chosen in the previous subsection.
		\begin{theorem}\label{thm:mainfull}
			Let $S \subset \RR^{2n+2}$ be a finite set of points sampled uniformly from the intersection $C' = C \cap [\BB_\epsilon(p)-\BB_{\epsilon_0}(p)^\circ]$. For any $\gamma \in (0,1)$, if the cardinality $\# S$ exceeds the bound $N_{C'}(\alpha,\gamma)$ from \eqref{eq:N}, then the following holds with probability exceeding $(1-\gamma)$: the set 
			\[
			S' := \set{y \in S \mid \text{\em dist}[y,\pi_\xi^{-1}(\delta)] < \alpha}
			\] consists of exactly $\ell$ nonempty point-clusters, each of diameter at most $2\alpha$, with the distance between distinct clusters exceeding $\mu - 2\alpha$.
		\end{theorem}
		\begin{proof}
			The intersection $C'$ is an embedded two-dimensional submanifold with boundary of $\RR^{2n+2}$ by Proposition \ref{prop:annmwb}, and $\alpha < \Delta_{C'}/2$ holds by \eqref{eq:alphabound}. Thus, we may safely apply Theorem \ref{thm:wangdense} to conclude that the set $S$ is $\alpha$-dense in $C'$ with probability exceeding $(1-\gamma)$. We will assume throughout the remainder of the argument that this density holds. 
			
			Since the inner radius $\epsilon_0$ is smaller than $\delta$, all points of the complex link $\set{x_1,\ldots,x_\ell}$ lie in $C'$. Let $B_i$ denote the open ball of radius $\alpha$ around each $x_i$. The inequalities which involve $\epsilon, \epsilon_0, \delta$ and $\kappa$ in  $\eqref{eq:alphabound}$ guarantee that  each $B_i$ is entirely contained within the annulus $[\BB_\epsilon(p)-\BB_{\epsilon_0}(p)]$. It follows from the $\alpha$-density of $S$ in $C'$ that the intersections $S_i := B_i \cap S$ are all nonempty; and moreover, the diameter of each $S_i$ is no larger than the diameter $2\alpha$ of $B_i$. We claim that these $S_i$ form the $\ell$ desired point clusters of $S'$. 
			
			To establish the claim, note from \eqref{eq:alphasep} that every point of $S'$ must lie in one of the $B_i$, whence $S' = \bigcup_{i=1}^\ell S_i$. Thus, it remains to show that the $S_i$ are separated from each other by a distance larger than $\mu-2\alpha$. To this end, note from \eqref{eq:mu} that the points $x_i$ and $x_j$ are separated by distance at least $\mu$ whenever $i \neq j$. Thus, points in distinct $B_i$ and $B_j$ are at least $\mu-2\alpha$ apart from each other, as desired.
		\end{proof}
		
		We know from \eqref{eq:alphabound} that $2\alpha$ is smaller than $\mu-2\alpha$, so the desired number $\ell = \e_pC$ can be determined with high confidence by clustering together points of $S'$ which lie within $2\alpha$ of each other and then counting the clusters.

		\bibliographystyle{abbrv}
		\bibliography{library}

	\end{document}